\documentclass[12pt]{amsart}

\usepackage{amsmath, amsthm, amssymb}
\usepackage{url} 
\usepackage{graphicx}
\usepackage{xcolor}
\usepackage{moreverb}
 \usepackage{fancyvrb}

\def\r{\mathbb R}

\usepackage{listings}
 \lstnewenvironment{code}[1][]
  {\lstset{#1}}
  {}

\newtheorem{theorem}{Theorem}[section]
 \newtheorem{proposition}[theorem]{Proposition}
 
\theoremstyle{definition}
\newtheorem{definition}[theorem]{Definition}
\newtheorem{example}[theorem]{Example}
\newtheorem{remark}[theorem]{Remark}

\begin{document}

\title[Translation surfaces with constant sectional curvature]{Classification of translation surfaces in $\r^3$ with constant sectional curvature }

\author{Muhittin Evren Aydin}
\address{Department of Mathematics, Faculty of Science, Firat University, Elazig,  23200 Turkey}
\email{meaydin@firat.edu.tr}
\author{Rafael L\'opez}
\address{Departamento de Geometr\'{\i}a y Topolog\'{\i}a Universidad de Granada 18071 Granada, Spain}
\email{rcamino@ugr.es}
\author{ Adela Mihai}
 \address{Technical University of Civil Engineering Bucharest,
Department of Mathematics and Computer Science, 020396, Bucharest, Romania
and Transilvania University of Bra\c{s}ov, Interdisciplinary Doctoral
School, 500036, Bra\c{s}ov, Romania}
 \email{adela.mihai@utcb.ro, adela.mihai@unitbv.ro}

\begin{abstract}
In this paper, we study translation surfaces in the Euclidean space   endowed with a canonical semi-symmetric non-metric connection. We completely classify the  translation surfaces of constant sectional curvature with respect to this connection, proving that they are   generalized cylinders. This consequence is the same as in the case of the Levi-Civita connection, but in this new setting,  there are also   generalized cylinders whose sectional curvature can be constant or non-constant, in contrast to the Levi-Civita connection,  where the Gaussian curvature is zero. 
\end{abstract}
\maketitle

\section{Introduction and statement of the results} \label{intro}

Let $(x,y,z)$ be the canonical coordinates in the $3$-dimensional Euclidean space $\r^3$. A surface $M$ in  $\r^3$ is called a {\it translation surface}  if $M$ is locally a graph of the form $z=f(x)+g(y)$, for some smooth functions $f$ and $g$. If $M$ is parametrized by  $\psi(x,y)=(x,y,f(x)+g(y))$, then $M$ is  the sum of two planar curves, namely, $\psi(x,y)=(x,0,f(x))+(0,y,g(y))$. Scherk proved that a  translation minimal surface is a plane or $z=\log\cos y-\log\cos x$ \cite{sc}. Liu  showed that   circular cylinders are the only translation surfaces with non-zero constant mean curvature \cite{li}. In \cite{li}, it was also proved that if the Gaussian curvature of a translation surface $M$ is constant then $M$ is a generalized cylinder and so the curvature is identically $0$. The notion of translation surfaces can be generalized assuming that the surface is the sum of two spatial curves of $\r^3$. It is known the classification of such surfaces if the Gaussian curvature is constant \cite{hr,lm} or if the mean curvature is zero \cite{ha,hr2,lp} or non-zero constant \cite{ha}.

In this paper, we study translation surfaces   when   $\r^3$ is endowed with a semi-symmetric non-metric connection. We introduce the definitions. Let $\langle,\rangle$ be the Euclidean metric of $\r^3$.   An affine connection $\widetilde{\nabla}$ in $\r^3$ is said to be {\it semi-symmetric} if there is a non-zero vector field $W\in\mathfrak{X}(\r^3)$ such that its torsion is
\begin{equation} \label{tors}
\widetilde{T}(X,Y)=\langle W,Y\rangle X-\langle W,X\rangle Y, \quad \forall X,Y \in \mathfrak{X}(\r^3).
\end{equation}
If in addition the metric $\langle,\rangle$ is not parallel, $\widetilde{\nabla}\langle,\rangle \neq0$, then  $\widetilde{\nabla}$ is said to be a {\it semi-symmetric non-metric connection}. In such a case,   there is a relation between $\widetilde{\nabla}$ and the Levi-Civita connection  $\widetilde{\nabla}^0$, namely,  
$$
\widetilde{\nabla}_XY=\widetilde{\nabla}^0_XY+\langle W,Y\rangle X, \quad \forall X,Y \in \mathfrak{X}(\r^3).
$$

The definition of a semi-symmetric connection on a Riemannian manifold   was introduced in \cite{fs} by   Friedmann and  Schouten. The study of submanifolds  endowed with semi-symmetric connections was initiated in \cite{hay} by   Hayden. Following this, the submanifolds in Riemannian manifolds endowed with a semi-symmetric metric connection were studied in \cite{im,na,ya}. The notion of semi-symmetric non-metric connections on Riemannian manifolds was introduced in \cite{ag,ac} by  Agashe and   Chafle, studying their properties and submanifolds.
  
Many semi-symmetric non-metric connections can be defined by depending on   the vector field $W$. The simplest choice for $W$ is to consider the canonical vector fields  $\{\partial_x,\partial_y,\partial_z\}$ of $\r^3$. Without loss of generality, in this paper we will assume that  $W=\partial_z$.

\begin{definition} A semi-symmetric non-metric connection $\widetilde{\nabla}$ on $\r^3$ is said to be {\it canonical} if it is defined by
\begin{equation} \label{nmc}
\widetilde{\nabla}_XY=\widetilde{\nabla}^0_XY+\langle \partial_z ,Y \rangle X , \quad X,Y \in \mathfrak{X}(\r^3).
\end{equation}
\end{definition}

To abbreviate, we will say simply {\it canonical snm-connection}. Once we have defined the connection $\widetilde{\nabla}$, we ask for the similar problems that were presented at the beginning of the paper. Let $M$ be a surface isometrically immersed in $\r^3$. Denote by $\nabla$ and $\nabla^0$ be the induced connections on $M$ from $\widetilde{\nabla}$ and $\widetilde{\nabla}^0$, respectively. For the problems involving the mean curvature, we point out that the second fundamental forms of $M$ with respect to $\widetilde{\nabla}$ and $\widetilde{\nabla}^0$ coincide (see Eq. \eqref{gauss-form} below). In particular, the mean curvature coincides for both connections and thus questions related with the mean curvature of surfaces are the same than with the Levi-Civita connection. 

We now deal with  problems about the sectional curvature. Let $K^0$ denote the sectional curvature of $M$ with respect to the induced Levi-Civita connection $\nabla^0$. Since the dimension of $M$ is $2$, the sectional curvature $K^0$ is simply the Gaussian curvature of $M$.  However, the notion of sectional curvature of $\widetilde{\nabla}$, and consequently of $\nabla$, cannot be defined by using the usual way as the Levi-Civita connection. For this, recently the third author of the present paper in collaboration with I. Mihai  have  defined the sectional curvature of a Riemannian manifold endowed with a snm-connection \cite{am0}. So, if $\widetilde{R}$ is the tensor curvature of $\widetilde{\nabla}$, then the sectional curvature $\widetilde{K}(\pi)$ of a plane $\pi$ of $\r^3$ is defined by 
$$
\widetilde{K}(\pi)=\frac{1}{2}\left( \tilde{g}(\widetilde{R}(e_1,e_2)e_2,e_1)+\tilde{g}(\widetilde{R}(e_2,e_1)e_1,e_2)\right),
$$
where $\{e_1,e_2\}$ is an orthonormal basis of $\pi$. This definition does not depend on the basis of $\pi$. Similarly, it is defined the sectional curvature $K$ of the tangent plane of a surface of $\r^3$ with respect to $\nabla$.

Motivated by the results of translation surfaces with constant Gauss curvature in the Euclidean space, we consider the problem of classifying the translation surfaces with   $K$ constant. We are able to provide a successful answer to this problem in the following result. 

\begin{theorem}\label{mr}
If $M$ is a  translation surface in $\r^3$ of constant sectional curvature $K$ with respect to the canonical snm-connection    then $M$ is a generalized cylinder.  
\end{theorem}

Since the canonical snm-connection $\widetilde{\nabla}$ is constructed with the vector field $\partial_z$, the coordinates $x$ and $y$ can be interchanged. Thus we can assume that a translation surface can be given in the forms $z=f(x)+g(y)$ or $x=f(y)+g(z)$. From Thm. \ref{mr}, if the surface is a generalized cylinder, then  one of the functions $f$ or $g$ must be linear. It is important to point out a difference with respect to the Levi-Civita connection. In $\r^3$, the Gaussian curvature of a generalized cylinder is constantly $0$. In contrast, with the canonical snm-connection $\nabla$, the sectional curvature $K$ may or may not be a constant  but if it is a constant, it is not necessarily $0$.

The organization of the paper is as follows. We first recall some fundamental formulas and equations in Sect. \ref{prel} for a Riemannian manifold and its submanifolds endowed with a snm-connection. In Sect. \ref{compt} we calculate the sectional curvature $K$ of the translation surfaces. Section \ref{main-proof} is devoted to prove  Thm. \ref{mr}. Once we know that the translation surfaces with $K$ constant are generalized cylinders, in Sect. \ref{cyln}, we obtain a complete description of these cylinders in Thms. \ref{t51}, \ref{t52} and \ref{t53}.
 
\section{Riemannian manifolds endowed with snm connections} \label{prel}

Let $(\widetilde{M},\tilde{g})$ be a Riemannian manifold, $\mbox{dim}\widetilde{M}\geq 2$ and let $\widetilde{\nabla}$ be an affine connection on $\widetilde{M}$. The torsion of $\widetilde{\nabla}$ is a $(1,2)$-tensor field $\widetilde{T}$ and the curvature of $\widetilde{\nabla}$ is a $(1,3)$-tensor field $\widetilde{R}$ defined, respectively, by
\begin{equation*}
\begin{split}
\widetilde{T}(X,Y)&=\widetilde{\nabla}_XY-\widetilde{\nabla}_YX-[X,Y], \quad \forall X,Y \in \mathfrak{X}(\widetilde{M}),\\
\widetilde{R}(X,Y)Z&=\widetilde{\nabla}_X\widetilde{\nabla}_YZ-\widetilde{\nabla}_Y\widetilde{\nabla}_XZ-\widetilde{\nabla}_{[X,Y]}Z, \quad \forall X,Y,Z \in \mathfrak{X}(\widetilde{M}).
\end{split}
\end{equation*}

The affine connection $\widetilde{\nabla}$ on $\widetilde{M}$ is said to be {\it semi-symmetric} if there is a nonzero vector field $W\in \mathfrak{X}(\widetilde{M})$ such that its torsion is
\begin{equation*}
\widetilde{T}(X,Y)=\tilde{g}(Y,W)X-\tilde{g}(W,X)Y, \quad \forall X,Y \in\mathfrak{X}(\widetilde{M}).
\end{equation*}
See \cite{fs}. We also call $\widetilde{\nabla}$ a {\it semi-symmetric non-metric connection} ({\it snm-connection}) if $\widetilde{\nabla}\tilde{g}\neq0$ \cite{ag,ac}.

Let  $\widetilde{\nabla}$ be a snm-connection on $(\widetilde{M},\tilde{g})$. If $\widetilde{\nabla}^0$  is the Levi-Civita connection, then   $\widetilde{\nabla}$ and $\widetilde{\nabla}^0$ are related by    
$$
\widetilde{\nabla}_XY=\widetilde{\nabla}^0_XY+\tilde{g}(W,Y)X, \quad \forall X,Y \in \mathfrak{X}(\widetilde{M}).
$$
 Also, there is a relation between $\widetilde{R}$ and the Riemannian curvature tensor $\widetilde{R}^0$ of $\widetilde{\nabla}^0$. For orthonormal vectors $e_1,e_2\in T_p\widetilde{M}$, $p\in\widetilde{M}$, we have 
$$
\tilde{g}(\widetilde{R}(e_1,e_2)e_2,e_1)=\tilde{g}(\widetilde{R}^0(e_1,e_2)e_2,e_1)-e_2(\tilde{g}(W,e_2))+\tilde{g}(W,\widetilde{\nabla}^0_{e_2}e_2)+\tilde{g}(W,e_2)^2. 
$$
Although the first term at the right side is the sectional curvature of the plane section $\pi=\text{span}\{e_1,e_2\}$ with respect to $\widetilde{\nabla}^0$, the term at the left side depends on the choice of the basis of $\pi$. Therefore,   $\tilde{g}(\widetilde{R}(e_1,e_2)e_2,e_1)$ does not stand for a sectional curvature of the plane $\pi$. The following geometric quantity  was proposed in \cite{am0} as the sectional curvature $\widetilde{K}$ with respect to $\widetilde{\nabla}$, 
\begin{equation}\label{kk}
\widetilde{K}(\pi)=\frac{1}{2}\left( \tilde{g}(\widetilde{R}(e_1,e_2)e_2,e_1)+\tilde{g}(\widetilde{R}(e_2,e_1)e_1,e_2)\right),
\end{equation}
which is proved to be independent of the orthonormal basis of $\pi$.

\begin{definition} \label{def-sec-cur}
Let  $\widetilde{\nabla}$ be   a snm-connection on $(\widetilde{M},\tilde{g})$. If  $p\in\widetilde{M}$ and  $\pi$ is a plane in $ T_p\widetilde{M}$, we call $\widetilde{K}(\pi)$   the {\it sectional curvature} of   $\pi$ with respect to $\widetilde{\nabla}$. 
\end{definition}

In case that $\{e_1,e_2\}$ is a basis of $\pi$, which it is not necessarily orthonormal, it is immediate that  
\begin{equation}\label{kk2}
\widetilde{K}(\pi)=\frac{\tilde{g}(\widetilde{R}(e_1,e_2)e_2,e_1)+\tilde{g}(\widetilde{R}(e_2,e_1)e_1,e_2)}{2( \tilde{g}(e_1,e_1)\tilde{g}(e_2,e_2)-\tilde{g}(e_1,e_2)^2 )}.
\end{equation}
 
 As an example,  we   particularize to the case that $\widetilde{M}$ is the Euclidean space $\r^3$ calculating the sectional curvature of planes with respect to the canonical snm-connection.
 \begin{proposition}
  Let $\widetilde{\nabla}$ be the canonical snm-connection  on $\r^3$ defined in  \eqref{nmc} and $\{\vec{u},\vec{v}\}$ an orthonormal basis of a plane $P$ in $\r^3$. Then the sectional curvature of $P$ is
$$
K(P)=K(\vec{u},\vec{v})=\frac{1}{2}(u_3^2+v_3^2),
$$
where $u_3$ and $v_3$ are the third components of $\vec{u}$ and $\vec{v}$ with respect to the canonical basis.
\end{proposition}
\begin{proof}
For the covariant derivatives, we use the formula in \eqref{nmc}, 
\begin{equation*}
\begin{array}{lll}
\widetilde{\nabla}_{\partial_x}\partial_x=0, &\widetilde{\nabla}_{\partial_x}\partial_y=0,&\widetilde{\nabla}_{\partial_x}\partial_z=\partial_x,\\
\widetilde{\nabla}_{\partial_y}\partial_x=0, &\widetilde{\nabla}_{\partial_y}\partial_y=0,&\widetilde{\nabla}_{\partial_y}\partial_z=\partial_y,\\
\widetilde{\nabla}_{\partial_z}\partial_x=0, &\widetilde{\nabla}_{\partial_z}\partial_y=0,&\widetilde{\nabla}_{\partial_z}\partial_z=\partial_z.\\
\end{array}
\end{equation*}
Set
$$
\vec{u}=u_1\partial_x+u_2\partial_y+u_3\partial_z, \quad \vec{v}=v_1\partial_x+v_2\partial_y+v_3\partial_z.
$$
Therefore we calculate 
\begin{equation*}
\begin{array}{lll}
\widetilde{\nabla}_{\vec{u}}\vec{u}=u_3\vec{u}, &\widetilde{\nabla}_{\vec{u}}\vec{v}=v_3\vec{u},\\
\widetilde{\nabla}_{\vec{v}}\vec{u}=u_3\vec{v}, &\widetilde{\nabla}_{\vec{v}}\vec{v}=v_3\vec{v}
\end{array}
\end{equation*}
and
\begin{equation*}
\begin{array}{lll}
\widetilde{\nabla}_{\vec{u}}\widetilde{\nabla}_{\vec{v}}\vec{v}=v_3^2\vec{u}, &\widetilde{\nabla}_{\vec{v}}\widetilde{\nabla}_{\vec{u}}\vec{v}=u_3v_3\vec{v},\\
\widetilde{\nabla}_{\vec{v}}\widetilde{\nabla}_{\vec{u}}\vec{u}=u_3^2\vec{v}, &\widetilde{\nabla}_{\vec{u}}\widetilde{\nabla}_{\vec{v}}\vec{u}=u_3v_3\vec{u}.
\end{array}
\end{equation*}
Moreover, it is obviously $[\vec{u},\vec{v}]=0$. In consequence, the curvature tensors of $\widetilde{\nabla}$ are
\begin{equation*}
\begin{array}{l}
\widetilde{R}(\vec{u},\vec{v})\vec{v}=v_3^2\vec{u}-u_3\vec{v},\\
\widetilde{R}(\vec{v},\vec{u})\vec{u}=u_3^2\vec{v}-u_3\vec{u}.
\end{array}
\end{equation*}
Thus 
$$\widetilde{g}(\widetilde{R}(\vec{u},\vec{v})\vec{v},\vec{u})=v_3^2,\quad \widetilde{g}(\widetilde{R}(\vec{v},\vec{u})\vec{u},\vec{v})=u_3^2.$$
This gives the result and concludes the proof.  
\end{proof}

As a conclusion, the sectional curvatures of the coordinate planes in $\r^3$ are
 $$
\widetilde{K}(\partial_x,\partial_y)=0, \quad \widetilde{K}(\partial_x,\partial_z)=\widetilde{K}(\partial_y,\partial_z)=\frac{1}{2}.
$$

Let   $M$ be a submanifold of $\widetilde{M}$ with $\text{dim}M>1$. The formulas of Gauss for $\widetilde{\nabla}$ and $\widetilde{\nabla}^0$ are, repsectively, 
\begin{equation} \label{gauss-form}
\begin{split}
\widetilde{\nabla}_XY&=\nabla_XY+h(X,Y),\\
  \widetilde{\nabla}^0_XY&=\nabla^0_XY+h^0(X,Y),
  \end{split}
\end{equation}
where $X,Y \in \mathfrak{X}(M)$,  $\nabla$ and $\nabla^0$ are the induced connections on $M$ and $h$ a $(0,2)$-tensor field on $M$.  We see that  $h$ coincides with the  second fundamental form $h^0$ of $M$ (\cite{ac}).

\begin{proposition} Let $M$ be a submanifold of $(\widetilde{M},\widetilde{g})$ and let $\widetilde{\nabla}$ be a snm-connection. Then $h=h^0$.
\end{proposition}
\begin{proof} From the definition of $\widetilde{\nabla}$, if $X,Y\in\mathfrak{X}(M)$, we have from the second equation of \eqref{gauss-form}
$$\widetilde{\nabla}_XY=\widetilde{\nabla}_X^0Y+\langle W,Y\rangle X=\nabla^0_XY+\langle W,Y\rangle X+h^0(X,Y).$$
By taking the first equation of \eqref{gauss-form} and comparing the tangent and normal parts with respect to $M$, we deduce 
\begin{equation*}
\begin{split}
\nabla_XY&=\nabla_X^0Y+\langle W,Y\rangle X,\\
h(X,Y)&=h^0(X,Y).
\end{split}
\end{equation*}
This proves the result. 
\end{proof}

Consider the decomposition of   the vector field $W$ that defines $\widetilde{\nabla}$ in its tangential and normal components with respect to   $M$,
$$
W=W^{\top}+W^{\perp}, \quad W^{\top}\in \mathfrak{X}(M), \quad W^{\perp}\in \mathfrak{X}(M)^{\perp}.
$$
Let $g$ be the induced metric   on $M$. If $X,Y,Z,U\in\mathfrak{X}(M)$, then the  Gauss equation with respect to $\widetilde{\nabla}$ is  (\cite{ac})
\begin{equation} \label{gauss-eq}
\begin{split}
g(R(X,Y)Z,U)&=\tilde{g}(\widetilde{R}(X,Y)Z,U)-\tilde{g}(h(X,Z),h(Y,U))\\
&+\tilde{g}(h(X,U),h(Y,Z))  -\tilde{g}(W^{\perp},h(Y,Z))g(X,U)\\
&+\tilde{g}(W^{\perp},h(X,Z))g(Y,U).
\end{split}%
\end{equation}

\section{Sectional curvature of translation surfaces}\label{compt}

Let $M$ be a   surface in $\r^3$. Since the dimension of $M$ is $2$, there is only one tangent plane and, consequently, we simply write $K$ to denote the sectional curvature of $M$.  In this section, we compute the sectional curvature $K$ of a translation surface with respect to the canonical snm-connection $\widetilde{\nabla}$. In principle, there are three types of translation surfaces depending on the choice of two coordinate planes. Since the canonical snm-connection $\widetilde{\nabla}$   utilizes the vector field $\partial_z$, then the roles of the coordinates $x$ and $y$ can be interchanged. Therefore, after renaming the coordinates $x$ and $y$, if necessary,  there are only two types of translations surfaces,  namely,  of type $z=f(x)+g(y)$ and   of type $x=f(y)+g(z)$.

\subsection{Translation surfaces of type $z=f(x)+g(y)$} 
Assume that $f\colon I\to\r$ and $g\colon J\to \r$ are two smooth functions. To simplify the notation, we will use $f'$, $f''$, $g'$ $g''$ and so on, understanding which is the  variable in each case.  A basis $\{e_1,e_2 \}$ of $\mathfrak{X}(M)$ is
$$
e_1=\partial_x+f'\partial_z, \quad e_2=\partial_y+g'\partial_z.
$$
The coefficients of the first fundamental form are
\begin{equation} \label{firs-coef}
E=1+f'^2, \quad F=f'g', \quad G=1+g'^2.
\end{equation} 

If we  set 
$$w=EG-F^2=1+f'^2+g'^2,$$
then the unit normal vector field $\xi$ of $M$ is
$$
\xi=\frac{1}{\sqrt{w}}(-f'\partial_x-g'\partial_y+\partial_z).
$$

By a direct computation from the definition of $\widetilde{\nabla}$ in \eqref{nmc}, the covariant derivatives are
\begin{equation*}
\begin{array}{ll}
\widetilde{\nabla}_{e_1}e_1=f'e_1+f''\partial_z, & \widetilde{\nabla}_{e_1}e_2=g'e_1,\\
\widetilde{\nabla}_{e_2}e_1=f'e_2 ,&
\widetilde{\nabla}_{e_2}e_2=g'e_2 +g''\partial_z.
\end{array}
\end{equation*}
From  \eqref{tors} the torsion is $\widetilde{T}(e_1,e_2)=g'e_1-f'e_2$ and hence we have 
$$
[e_1,e_2]=\widetilde{\nabla}_{e_1}e_2-\widetilde{\nabla}_{e_2}e_1-\widetilde{T}(e_1,e_2)=0.
$$

The coefficients of the $(0,2)$-tensor field $h$ given in the formula of Gauss coincide with the second fundamental form $h^0$ of the Levi-Civita connection $\widetilde{\nabla}^0$. Thus 
\begin{equation} \label{sec-coef}
h_{11}=\frac{f''}{\sqrt{w}}, \quad h_{12}=h_{21}=0, \quad h_{22}=\frac{g''}{\sqrt{w}}.
\end{equation} 
The covariant derivatives of second order are
\begin{equation*}
\begin{array}{ll}
\widetilde{\nabla}_{e_1}\widetilde{\nabla}_{e_2}e_2=( g''+g'^2 ) e_1,&\widetilde{\nabla}_{e_2}\widetilde{\nabla}_{e_1}e_1=( f''+f'^2 ) e_2,\\
\widetilde{\nabla}_{e_2}\widetilde{\nabla}_{e_1}e_2=g''e_1+f'g'e_2 ,&
\widetilde{\nabla}_{e_1}\widetilde{\nabla}_{e_2}e_1=f'g'e_1 +f''e_2.
\end{array}
\end{equation*}

The curvature tensor $\widetilde{R}$ necessary for the computation of $\widetilde{K}$ is
\begin{eqnarray*}
\widetilde{R}(e_1,e_2)e_2&=&g' ( g'e_1-f'e_2), \\
 \widetilde{R}(e_2,e_1)e_1&=&-f'(g'e_1-f'e_2).
\end{eqnarray*}
Thus
\begin{equation*}
\begin{split}
\langle \widetilde{R}(e_1,e_2)e_2, e_1 \rangle &=g'^2, \\
\langle \widetilde{R}(e_2,e_1)e_1, e_2 \rangle &=f'^2.
\end{split}
\end{equation*}

The Gauss equation given by \eqref{gauss-eq} is written by
\begin{equation*}
\begin{split}
\langle R(e_1,e_2)e_2, e_1 \rangle &=\langle \widetilde{R}(e_1,e_2)e_2, e_1 \rangle 
+h_{11}h_{22}-\frac{h_{22}}{\sqrt{w}}E \\
&=g'^2+\frac{g''(f''-f'^2-1)}{w},\\
\langle R(e_2,e_1)e_1, e_2 \rangle &=\langle \widetilde{R}(e_2,e_1)e_1, e_2 \rangle 
+h_{11}h_{22}-\frac{h_{11}}{\sqrt{w}}G\\
&=f'^2+\frac{f''(g''-g'^2-1)}{w}.
\end{split}
\end{equation*}

Since $\{e_1,e_2\}$ is not an orthonormal basis, we use formula \eqref{kk2} for the computation of the sectional curvature $K$ of $M$, obtaining 
\begin{equation*}
\begin{split}
K=K(e_1,e_2)&=\frac{\langle R(e_1,e_2)e_2,e_1\rangle+\langle R(e_2,e_1)e_1,e_2)\rangle}{2w}\\
&=\frac{f'^2+g'^2}{2w}+\frac{g''(f''-f'^2-1)+f''(g''-g'^2-1)}{2w^2}.
\end{split}
\end{equation*}
This identity writes as
\begin{equation} \label{sec-curv-type1}
 2w^2K=w(f'^2+g'^2) +
f''(g''-g'^2-1)+g''(f''-f'^2-1).
\end{equation}

\subsection{Translation surfaces of type $x=f(y)+g(z)$}   We do calculations in analogy to the previous case. A basis $\{ e_1,e_2\}$ of $\mathfrak{X}(M)$ is
$$
e_1=f'\partial_x+\partial_y, \quad e_2=g'\partial_x+\partial_z,
$$
and let $w=EG-F^2=1+f'^2+g'^2$. The unit normal   of $M$ is
$$
\xi=\frac{1}{\sqrt{w}}(\partial_x-f'\partial_y-g'\partial_z).
$$
The covariant derivatives are given by
\begin{equation*}
\begin{array}{ll}
\widetilde{\nabla}_{e_1}e_1=f''\partial_x, &\widetilde{\nabla}_{e_1}e_2=e_1 ,\\
\widetilde{\nabla}_{e_2}e_1=0 ,&\widetilde{\nabla}_{e_2}e_2=e_2+g''\partial_x.
\end{array}
\end{equation*}
The torsion is computed by $\widetilde{T}(e_1,e_2)=e_1$ and so the bracket is $[e_1,e_2]=0.$  Also, we have
\begin{equation*}
\begin{array}{ll}
\widetilde{\nabla}_{e_1}\widetilde{\nabla}_{e_2}e_2=e_1,&\widetilde{\nabla}_{e_2}\widetilde{\nabla}_{e_1}e_2=0,\\
\widetilde{\nabla}_{e_2}\widetilde{\nabla}_{e_1}e_1=0,&\widetilde{\nabla}_{e_1}\widetilde{\nabla}_{e_2}e_1=0.
\end{array}
\end{equation*}
Hence, we obtain
\begin{eqnarray*}
\widetilde{R}(e_1,e_2)e_2&=&e_1, \\  
\widetilde{R}(e_2,e_1)e_1&=&0.
\end{eqnarray*}

By the Gauss equation we get
\begin{equation*}
\begin{split}
\langle R(e_1,e_2)e_2, e_1 \rangle &=\langle \widetilde{R}(e_1,e_2)e_2, e_1 \rangle 
+h_{11}h_{22}-\frac{h_{22}}{\sqrt{w}}E \\
&=1+f'^2+\frac{g''(f''-f'^2-1)}{w},\\
\langle R(e_2,e_1)e_1, e_2 \rangle &=\langle \widetilde{R}(e_2,e_1)e_1, e_2 \rangle 
+h_{11}h_{22}-\frac{h_{11}}{\sqrt{w}}G \\
&=\frac{f''(g''-g'^2-1)}{w}.
\end{split}
\end{equation*}
Then, the sectional curvature $K$ of $M$ satisfies the equation
\begin{equation} \label{sec-curv-type2}
2w^2K=w(1+f'^2)+f''(g''-g'^2-1)+g''( f''-f'^2-1).
\end{equation}

\section{Proof of Theorem \ref{mr}}\label{main-proof}

Suppose that  $M$ is a translation surface in $\r^3$ with constant sectional curvature $K=K_0/2$. We will prove that $f$ or $g$ is a linear function, which implies that $M$ is a generalized cylinder. The proof of the theorem is by contradiction. Suppose that $f''g''\not=0$ and $f'g'\not=0$ in a subdomain of $I\times J$. Because $M$ can be given in two different forms, we separate the proof in two cases. 

\subsection{Case $z=f(x)+g(y)$.}

Since $f'^2+g'^2=w-1$, then Eq. \eqref{sec-curv-type1} writes as 
\begin{equation} \label{Kconst-00}
w^2K_0=w(w-1) +f''(g''-g'^2-1)+g''(f''-f'^2-1).
\end{equation}
Since the roles of $f$ and $g$ in  \eqref{Kconst-00} are symmetric, the arguments for   one of $f$ and $g$ are also valid for the other.  We discuss  two subcases:

{\bf Subcase  $K_0=1$}. Then   \eqref{Kconst-00} is 
$$w=2f''g''-(1+g'^2)f''-(1+f'^2)g''.$$
Dividing by $f''g''$, we write
\begin{equation} \label{Kconst-02}
\frac{w}{f''g''} +\frac{1+g'^2}{g''}+\frac{1+f'^2}{f''}=2.
\end{equation}
Differentiating  identity \eqref{Kconst-02} with respect to $x$ and next with respect to $y$, we have
$$\left(\frac{w}{f''g''}\right)_{xy}=0.$$
This yields
\begin{equation} \label{Kconst-03}
2f'f''^2g'''+2g'g''^2f'''-wg'''f'''=0.
\end{equation}
If $f'''=0$ and $g'''=0$ identically, then   $f''=c_1^{-1}$ and $g''=c_2^{-1}$, for nonzero constants $c_1,c_2$. This allows us to write   Eq. \eqref{Kconst-02} as  
$$
c_1c_2+c_1+c_2-2+c_1(c_2+1)f'^2+c_2(c_1+1)g'^2=0.
$$
It implies $c_1+1=c_2+1=c_1c_2+c_1+c_2-2=0$, which is not possible. 

Hence we have $f'''g'''\neq0$. If we go back to Eq. \eqref{Kconst-03}, then 
$$
1=\frac{2f'f''^2}{f'''}-f'^2+\frac{2g'g''^2}{g'''}-g'^2.
$$
There are some nonzero constants $c_3$ and $c_4$ such that $c_3+c_4=1$ and
$$
\frac{2f'f''^2}{f'''}-f'^2=c_3, \quad \frac{2g'g''^2}{g'''}-g'^2=c_4,
$$
or equivalently
$$
\frac{2f'f''}{c_3+f'^2}=\frac{f'''}{f''}, \quad \frac{2g'g''}{c_4+g'^2}=\frac{g'''}{g''}.
$$
Integrating once, we get
$$
f''=c_5(c_3+f'^2), \quad g''=c_6(c_4+g'^2),
$$
for some nonzero constants $c_5$ and $c_6$. Substituting into \eqref{Kconst-02}, we obtain a polynomial equation of $f'$ and $g'$   of the form
$$
A_0+A_1f'^2+A_2g'^2+A_3f'^2g'^2=0,
$$
where  
\begin{equation*}
\begin{split}
A_0&=1+c_3c_5+c_4c_6-2c_3c_4c_5c_6, \\
A_1&=1+c_5+c_4c_6-2c_4c_5c_6, \\
A_2&=1+c_3c_5+c_6-2c_3c_5c_6, \\
A_3&=c_5+c_6-2c_5c_6.
\end{split}
\end{equation*}
Since all coefficients $A_i$ must vanish, then the combinations  $A_0-A_1=0$ and $A_0-A_2=0$ give $c_3c_5=\frac12=c_4c_6.$ On the other hand,  
$$
0=A_1+A_2=3+c_5+c_6-2c_5c_6=3+A_3.
$$
This gives $A_3=-3$, a contradiction.

{\bf Subcase $K_0\neq1$}. We introduce new functions 
$$P=P(f)=f'^2,\quad Q=Q(g)=g'^2.$$
 Then $P'=2f''$ and $Q'=2g''$.   Equation \eqref{Kconst-00} is now
\begin{equation} \label{Kconst-04}
\left. 
\begin{array}{l}
4\left( K_0+(2K_0-1)(P +Q )+(K_0-1)(P^2+Q^2)+2(K_0-1)PQ \right ) \\
=P' (Q'-2Q-2 )+Q'(P'-2P-2).
\end{array}%
\right. 
\end{equation}
Obviously, $P''$ cannot be zero identically, because otherwise \eqref{Kconst-04} would be a polynomial on $P$ of degree $2$ whose leading coefficient is $4(K_0-1)\not=0$. By the symmetry, $Q'' \neq 0$. Taking successive   derivatives of identity \eqref{Kconst-04} with respect to $f$ and $g$, we obtain
\begin{equation} \label{Kconst-05}
8(K_0-1)P'Q'=P''(Q''-2Q' )+Q''(P''-2P').
\end{equation}
Claim:  $P''$   is not proportional to $P'$. The proof of the claim is by contradiction. If $P''=c_1P'$, $c_1\neq0$, then a first integration follows $P'=c_1P+d_1$, for some constant $d_1$. Inserting in \eqref{Kconst-04} we have a polynomial on $P$ with non-vanishing leading coefficient, which is a contradiction. By the symmetry we also conclude that $Q''$ is not proportional to $Q'$. 

Once we have established the claim, let us divide    \eqref{Kconst-05} by $P''Q''$, obtaining
$$
8(K_0-1)\frac{P'Q'}{P''Q''}=P''+Q''-2(P'+Q').
$$
Finally, the successive   derivatives with respect to $f$ and $g$ yield
$$
\left(\frac{P'}{P''}\right)'\left(\frac{Q'}{Q''}\right)'=0,
$$
which implies $P'/P''$ or $Q'/Q''$ are constant functions, contrary to the Claim. This finishes the proof when the translation surface is $z=f(x)+g(y)$.

\subsection{Case $ x=f(y)+g(z)$.} 

Equation \eqref{sec-curv-type2} is now
\begin{equation} \label{Kconst-10}
w^2K_0=w(1+f'^2)+f''(g''-g'^2-1)+g''(f''-f'^2-1).
\end{equation}
As opposed to the previous subsection, the roles of $f$ and $g$ in   \eqref{Kconst-10} are not symmetric.  Dividing  \eqref{Kconst-10} by $1+f'^2$, we have 
\begin{equation} \label{Kconst-11}
w(K_0-1)+K_0g'^2+g''+K_0\frac{g'^4}{1+f'^2} = 
\frac{f''(2g''-g'^2-1)}{1+f'^2}.
\end{equation}
We discuss three subcases:

{\bf Subcase   $K_0=0$}. Equation \eqref{Kconst-11} is
\begin{equation} \label{Kconst-12}
-1-f'^2-g'^2+g'' =\frac{f''(2g''-g'^2-1)}{1+f'^2}.
\end{equation}
If we differentiate  \eqref{Kconst-12} with respect to $y$ and $z$, one gets
$$
0=\left(\frac{f''}{1+f'^2}\right)'(2g''-g'^2)'.
$$
If $(\frac{f''}{1+f'^2})'=0$ identically, then \eqref{Kconst-12} would be a polynomial on $f'$ of degree $2$ with non-vanishing leading coefficient, which is not possible. Then $(2g''-g'^2)'=0$ identically, that is,  $2g''=c+g'^2$, for some constant $c$. Replacing in Eq.  \eqref{Kconst-12} we have   a polynomial on $g'$ of degree $2$ with non-vanishing leading coefficient. This is a contradiction again.

{\bf Subcase  $K_0=1$}. Then, \eqref{Kconst-11} is
\begin{equation} \label{Kconst-13}
g'^2+g''+\frac{g'^4}{1+f'^2} = 
\frac{f''(2g''-g'^2-1)}{1+f'^2}.
\end{equation}
We differentiate   \eqref{Kconst-13} with respect to $y$ and $z$, obtaining
$$
-8g'^3g''\frac{f'f''}{(1+f'^2)^2}=(2g''-g'^2)'\left (\frac{f''}{1+f'^2} \right )'.
$$
 Therefore there are    nonzero constants $c_1,c_2$ such that $c_1c_2=1$ and
$$
\left (\frac{f''}{1+f'^2} \right )'=c_1\left (\frac{1}{1+f'^2} \right )', \quad (2g''-g'^2)'=c_2(g'^4)' .
$$
Integrating both equations, we deduce
$$
f''=c_1+d_1(1+f'^2), \quad 2g''=c_2g'^4+g'^2+d_2,
$$
for some constants $d_1,d_2$. Inserting in \eqref{Kconst-13}, we obtain
$$
\frac{d_2}{2}+\frac{3}{2}g'^2+\frac{c_2}{2}g'^4 = 
\frac{c_1(d_2-1)}{1+f'^2}+d_1(c_2g'^4+d_2-1).
$$
This is a polynomial on $g'$ whose coefficient of the term of degree $2$ is $\frac32$: a contradiction.

{\bf Subase  $K_0\notin \{0,1\}$}. By the successive   derivatives of \eqref{Kconst-11} with respect to $y$ and $z$, we obtain
$$
K_0\left (\frac{1}{1+f'^2} \right)'(g'^4)' =\left (\frac{f''}{1+f'^2}\right)'(2g''-g'^2)'.
$$
In analogy to the previous subcase we may conclude
$$
f''=c_1+d_1(1+f'^2), \quad 2g''=c_2g'^4+g'^2+d_2.
$$
Substituting into \eqref{Kconst-11}, we have a polynomial equation on $f'$ of degree $4$ with leading coefficient $K_0-1\not=0$, which is not possible.
 
\section{Generalized cylinders with constant sectional curvature} \label{cyln}

In Thm. \ref{mr} we have proved that if $M$ is a translation surface $z=f(x)+g(y)$ or $x=f(y)+g(z)$ with    $K$ constant, then one of the functions $f$ or $g$ is linear. In particular, $M$ is a generalized cylinder whose rulings, depending on the case, are parallel to one of the coordinate axes. In this section we complete the classification of the translation surfaces with   $K$ constant giving a full description of these cylinders. 

  For the surface $z=f(x)+g(y)$, we will assume, without loss of generality, that the function $g$ is linear. After a vertical translation, which it does not affect the value of $K$, let $g(y)=ay$. However, as previously mentioned, in case $x=f(y)+g(z)$ the roles of $f$ and $g$ are not symmetric. In particular, we have to distinguish if $f$ or $g$ are linear functions. 
  
  \subsection{Generalized cylinders $z=f(x)+ay$}

Let $M$ be a generalized cylinder of type $z=f(x)+ay$ and assume that the sectional curvature of $M$ is  constant, $K=K_0/2$. 
\begin{theorem} \label{t51}
If a generalized cylinder in $\r^3$ of type $z=f(x)+ay$ has constant sectional curvature $K_0/2$ with respect to the canonical snm-connection, then it is either a plane $(K_0=\frac{a^2}{1+a^2})$, or a grim reaper cylinder $(K_0=1)$, or the function $f$ satisfies the integral equation
\begin{equation}\label{f51}
x=\pm \int^{f}\left(\frac{cK_0-e^{\frac{2u}{1+a^2}}-c}{e^{\frac{2u}{1+a^2}}-cK_0}\right)^{1/2}du, \quad K_0\in (0,1) \cup (1,\infty), \quad c>0,
\end{equation}
where the maximal domain is
\begin{enumerate}
\item $f\in(-\infty,(1+a^2)\log\sqrt{cK_0})$ if $0<K_0<1$,

\item $f\in ((1+a^2)\log \sqrt{c(K_0-1)},(1+a^2)\log \sqrt{cK_0})$ if $K_0>1$.
\end{enumerate}
\end{theorem}

\begin{proof} 
Equation   \eqref{Kconst-00} is now
\begin{equation} \label{cyl-k0-eq-1}
(1+a^2+f'^2)^2(K_0-1)+(1+a^2+f'^2)+(1+a^2)f''=0.
\end{equation}
We distinguish the following cases:
\begin{enumerate}
\item If $f'=0$ identically, then \eqref{cyl-k0-eq-1} is simply $(1+a^2)K_0=a^2$   and the resulting surface is a plane parallel with $K=\frac{a^2}{1+a^2}$.
\item If $K_0=1$, then Eq. \eqref{cyl-k0-eq-1} is 
$$(1+a^2)f''+1+a^2+f'^2=0. $$
The solution of this equation is
$$
f(x)=(1+a^2)\log(\cos(\frac{x+c}{\sqrt{1+a^2}}))+d, \quad c,d\in\r.
$$
This curve is known as grim reaper and the resulting surface $M$ becomes a grim reaper cylinder \cite{lop-grim}. 
\item Suppose $f''\not=0$ and  $K_0\not=1$. We introduce $P(f)=f'^2$. Denote by $p(f)=P(f)+a^2$, then \eqref{cyl-k0-eq-1} is  
$$
(1+a^2)p'+2(K_0-1)p^2 + 2(2K_0-1)p+2K_0=0.
$$
A first integral follows
\begin{equation} \label{p(f)-int-1}
f'^2+a^2=\frac{e^{\frac{2f}{1+a^2}}-cK_0}{cK_0-e^{\frac{2f}{1+a^2}}-c}, \quad   c>0,
\end{equation}
where $K_0$ must be always positive because   the left hand side is strictly  positive. In addition, the numerator and the denominator must be always negative, namely
$$
c(K_0-1)<e^{\frac{2f}{1+a^2}}<cK_0, \quad K_0\in (0,1) \cup (1,\infty).
$$
Consequently, we have the restriction on the domain
$$
\left\{ 
\begin{array}{l}
f\in(-\infty,(1+a^2)\log\sqrt{cK_0}), \quad 0<K_0<1  \\ 
f\in ((1+a^2)\log \sqrt{c(K_0-1)},(1+a^2)\log \sqrt{cK_0}), \quad K_0>1.
\end{array}%
\right. 
$$
Finally, by integrating  the differential equation \eqref{p(f)-int-1}, we obtain \eqref{f51}.
\end{enumerate}

\end{proof}
\begin{remark} As it was indicated in Introduction, generalized cylinders may have constant sectional curvature but not necessarily $0$. In the cylinders of Thm. \ref{t51} the constant $K$ is non-negative. The case $K=0$ appears with the value $a=0$.
\end{remark}

In particular cases, the integral \eqref{f51} can be explicitly obtained. See Fig. \ref{fig01}.
\begin{example}
Suppose  $a=0$. 
\begin{enumerate}
\item If   $c=2$ and $K_0=1/2$, then the solution of \eqref{f51} is
\begin{equation}\label{int-cyl-01}
x=\pm\frac{1}{2}\left(\sin ^{-1}\left(e^{2 f}\right)-\tanh ^{-1}(\sqrt{1-e^{4 f}})\right )+d, \quad d\in\r,
\end{equation}
where $f\in(-\infty,0)$.
\item If  $c=3$ and $K_0=2$, then the solution of \eqref{f51} is  
\begin{equation}\label{int-cyl-02}
x=\pm \left(\sin ^{-1} (  \sqrt{\frac{e^{2f}}{3}-1}) - \frac{1}{\sqrt{2}}%
\tan ^{-1} (\sqrt{\frac{2e^{2f}-6}{6-e^{2f}}})\right ) +d, \quad d\in\r,
\end{equation}
where $f\in(\log \sqrt{3},\log \sqrt{6})$.
\end{enumerate}
\end{example}

\begin{figure}[hbtp]
\begin{center}
\includegraphics[width=.25\textwidth]{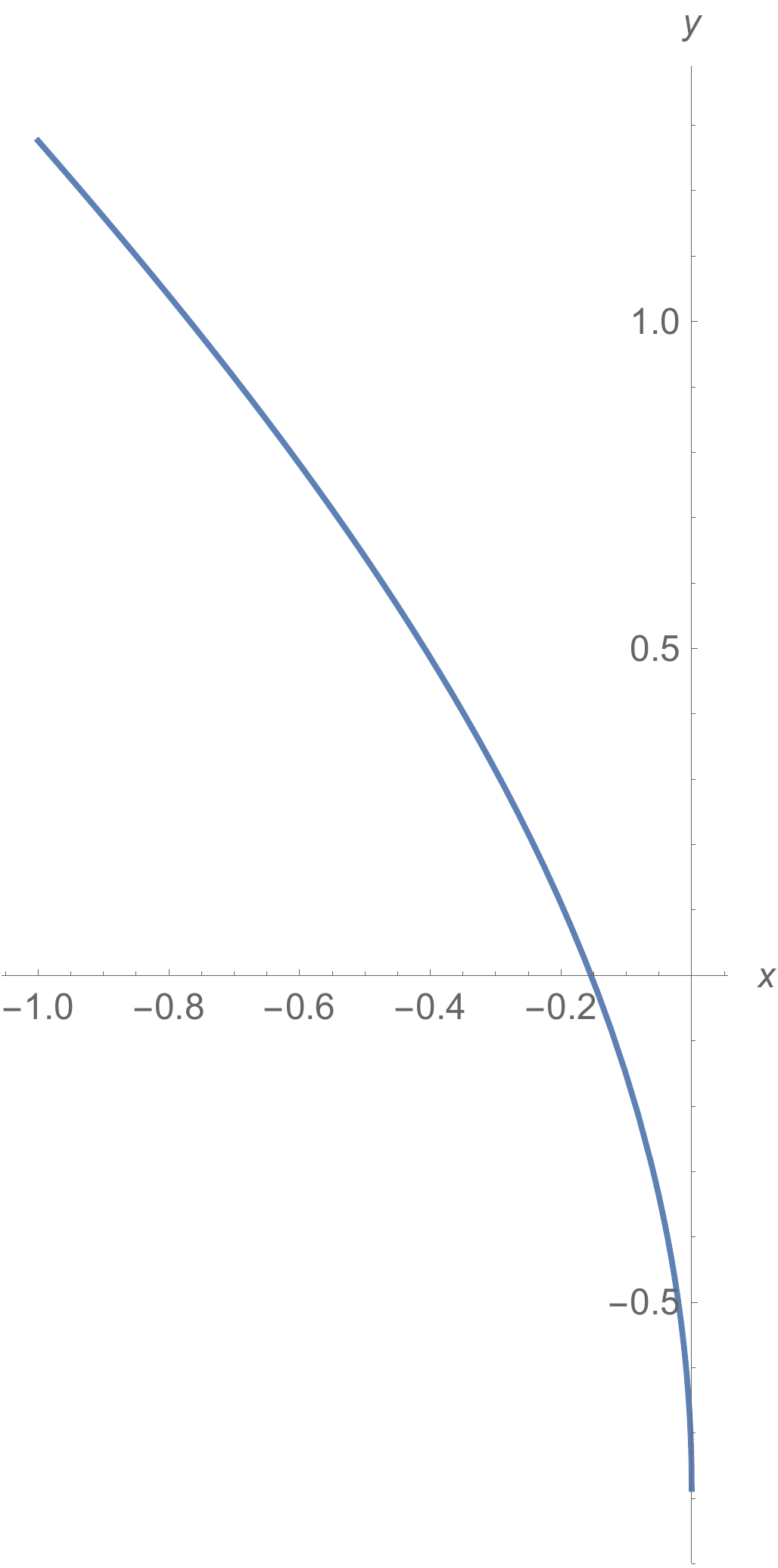}\qquad \qquad \includegraphics[width=.25\textwidth]{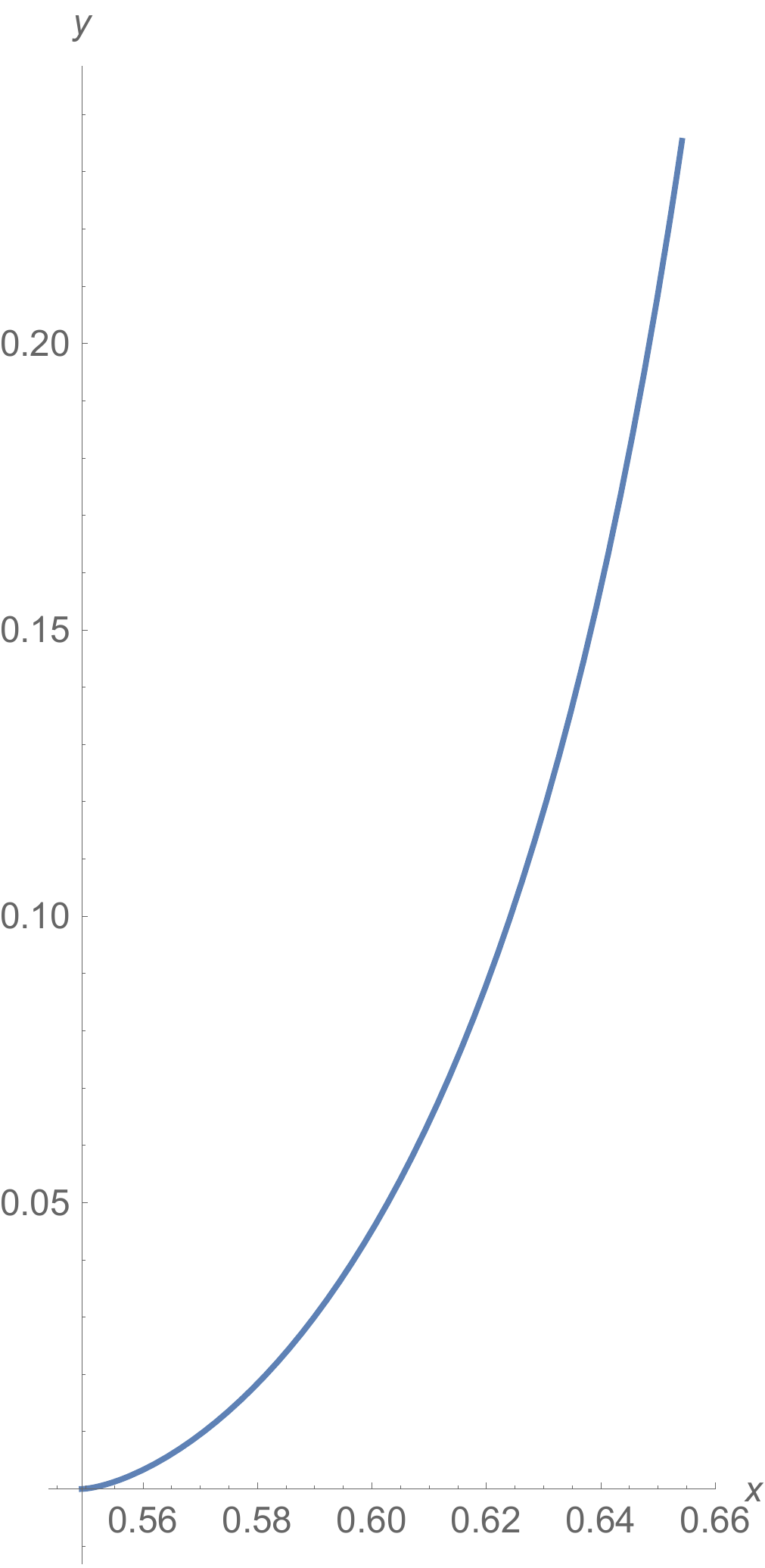}
\end{center}
\caption{Graphs of   \eqref{int-cyl-01} (left) and \eqref{int-cyl-02} (right).}\label{fig01}
\end{figure}

\subsection{Generalized cylinders $x=f(y)+az$}

Let  $M$ be  a generalized cylinder of type  $x=f(y)+az$ and assume that the sectional curvature of $M$ is  constant, $K=K_0/2$. Therefore, \eqref{Kconst-11} is now
\begin{equation} \label{cyl-k0-eq-2}
(1+a^2+f'^2)(K_0-1)+a^2K_0+K_0\frac{a^4}{1+f'^2}+(1+a^2)\frac{f''}{1+f'^2}=0.
\end{equation}
We distinguish the following two cases.
\begin{enumerate}
\item Case $a=0$. Equation \eqref{cyl-k0-eq-2} is now
$$
(1+f'^2)(K_0-1)+\frac{f''}{1+f'^2}=0.
$$
If $K_0=1$, then $f''=0$, or equivalently $M$ becomes a plane of sectional curvature $1/2$. If $K_0\neq1$, then we may write
$$
\frac{f''}{(1+f'^2)^2}=1-K_0.
$$
By integrating, we have
\begin{equation} \label{int21}
f'^2=-\frac{1+c+2(1-K_0)f}{c+2(1-K_0)f}, \quad c\in \r,
\end{equation}
where the numerator is positive and the denominator is negative. Also, the domain is
$$
\left\{ 
\begin{array}{l}
f\in (\frac{c+1}{2(K_0-1)}, \frac{c}{2(K_0-1)} ), \quad K_0<1  \\ 
f\in (\frac{c}{2(K_0-1)}, \frac{c+1}{2(K_0-1)}) , \quad K_0>1.
\end{array}%
\right. 
$$
It follows from  \eqref{int21} that
\begin{equation} \label{sol21}
y=\pm \int^f \left (-\frac{c+2(1-K_0)u}{1+c+2(1-K_0)u}\right)^{1/2}du.
\end{equation}
\item Case $a\neq 0$. We have subcases:
\begin{enumerate}
\item If $K_0=0$, then   \eqref{cyl-k0-eq-2} yields 
$$-(1+a^2+f'^2)  +(1+a^2)\frac{f''}{1+f'^2}=0.$$
With the change $P(f)=f'^2$, this equation writes as 
$$\frac{P'}{(1+a^2+P)(1+P)}=\frac{2}{1+a^2}.$$
This gives a first integral 
\begin{equation} \label{int22}
f'^2=\frac{ce^{\frac{2a^2f}{1+a^2}}-1+ca^2e^{\frac{2a^2f}{1+a^2}}}{1-ce^{\frac{2a^2f}{1+a^2}}},\quad c>0,
\end{equation}
where the domain is
$$
f\in\left (-\frac{1+a^2}{2a^2}\log c(1+a^2),-\frac{1+a^2}{2a^2}\log c \right). 
$$
Integrating   \eqref{int22} we have
\begin{equation} \label{sol22}
y=\pm\int^f \left ( \frac{1-ce^{\frac{2a^2u}{1+a^2}}}{-1+c(1+a^2)e^{\frac{2a^2u}{1+a^2}}} \right ) ^{1/2}du.
\end{equation}
\item If  $K_0=1$, we have from Eq. \eqref{cyl-k0-eq-2} that
$$(1+a^2)f''+a^2f'^2+a^2(1+a^2)=0.$$
The solution of this equation is
$$f(y)=c_1+\frac{1+a^2}{a^2}\log(\cos(\frac{a^2}{\sqrt{1+a^2}}y+c_2)),\quad c_1,c_2\in\r.$$
The curve is the grim reaper again. 
\item Suppose next that $K_0\notin \{0,1\}$. If  $P(f)=f'^2$, then   \eqref{cyl-k0-eq-2} writes as
$$
(1+a^2+P)(K_0-1)+a^2K_0+K_0\frac{a^4}{1+P}+(1+a^2)\frac{P'}{2(1+P)}=0.
$$
Integrating, one obtains
$$
f'^2(1+(K_0-1)c e^{-\frac{2a^2 f}{1+a^2}})=(1-K_0(1+a^2))ce^{  -\frac{2a^2 f}{1+a^2}}+1+a^2, \quad c>0,
$$
or equivalently
\begin{equation} \label{int23}
f'^2=-(1+a^2)\frac{(K_0- \frac{1}{1+a^2})c e^{-\frac{2a^2 f}{1+a^2}}-1}{(K_0-1)c e^{-\frac{2a^2 f}{1+a^2}}+1}.
\end{equation}
The domain is
$$
\left\{ 
\begin{array}{l}
f\in (\frac{1+a^2}{2a^2}\log c(1- K_0),\infty ), \quad -\infty <K_0\leq\frac{1}{1+a^2},  \\ 
f\in (-\infty,\frac{1+a^2}{2a^2}\lambda) \text{ or } f\in (\frac{1+a^2}{2a^2}\mu,\infty), \quad \frac{1}{1+a^2} <K_0<1, \\ 
f\in (\frac{1+a^2}{2a^2}\log c( K_0 - \frac{1}{1+a^2} ),\infty ), \quad K_0>1, 
\end{array}%
\right. 
$$
where 
$$
\left.
\begin{array}{l}
\lambda \in \text{min}\{\log c( K_0 - \frac{1}{1+a^2} ),\log c(1- K_0)\},  \\ 
\mu \in \text{max}\{\log c( K_0 - \frac{1}{1+a^2} ),\log c(1- K_0)\}.
\end{array}%
\right. 
$$
Integrating \eqref{int23}, one gets
\begin{equation} \label{sol23}
y=\pm \sqrt{(1+a^2)} \int^f  \left (-\frac{(K_0-1)c +e^{\frac{2a^2 u}{1+a^2}}}{(K_0- \frac{1}{1+a^2})c -e^{\frac{2a^2 u}{1+a^2}}} \right )^{1/2}du.
\end{equation}
\end{enumerate}
\end{enumerate}

\begin{theorem} \label{t52}
If a generalized cylinder in $\r^3$ of type $x=f(y)+az$ has constant sectional curvature with respect to the canonical snm-connection, then it is either a plane of sectional curvature $\frac12$, or a grim reaper cylinder, or one of \eqref{sol21}, \eqref{sol22} and \eqref{sol23} holds.
\end{theorem}

In analogy to the previous subsection, for some particular values of $c$ and $K_0$,  the integrals in Theorem \ref{t52} can be explicitly obtained. 

\begin{example}
Fix $a=0$. See Fig. \ref{fig02}.
\begin{enumerate}
\item If $c=-2$ and $K_0=-1$, then solution of the integral \eqref{sol21} is
\begin{equation}\label{int-cyl-11}
y=\mp \frac{\sqrt{1-2 f} \sqrt{4 f-1}}{2 \sqrt{2}}\pm \frac{1}{4} \sin ^{-1}\left(\sqrt{2-4 f}\right), \quad d\in\r,
\end{equation}
where $f\in(\frac{1}{4},\frac{1}{2})$. 
\item For $c=0$ and $K_0=2$, the solution of \eqref{sol21} is
\begin{equation}\label{int-cyl-12}
y=\mp \frac{\sqrt{1-2 f} \sqrt{f}}{\sqrt{2}}\pm \frac{1}{2} \sin ^{-1}\left(\sqrt{2} \sqrt{f}\right)+d, \quad d\in\r,
\end{equation}
where $f\in(0,\frac{1}{2})$. 
\end{enumerate}
\end{example}

\begin{figure}[hbtp]
\begin{center}
\includegraphics[width=.25\textwidth]{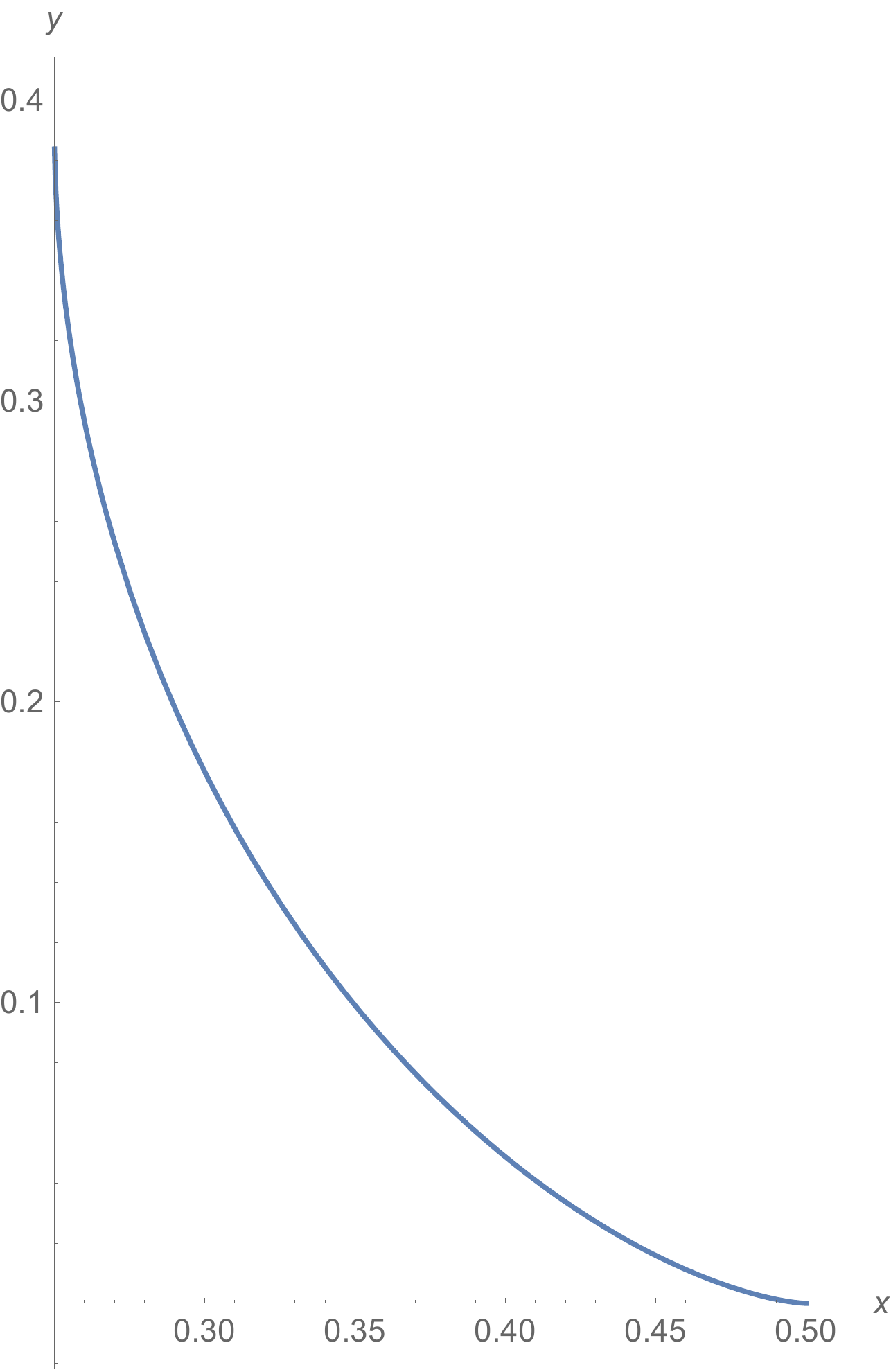}\qquad \includegraphics[width=.25\textwidth]{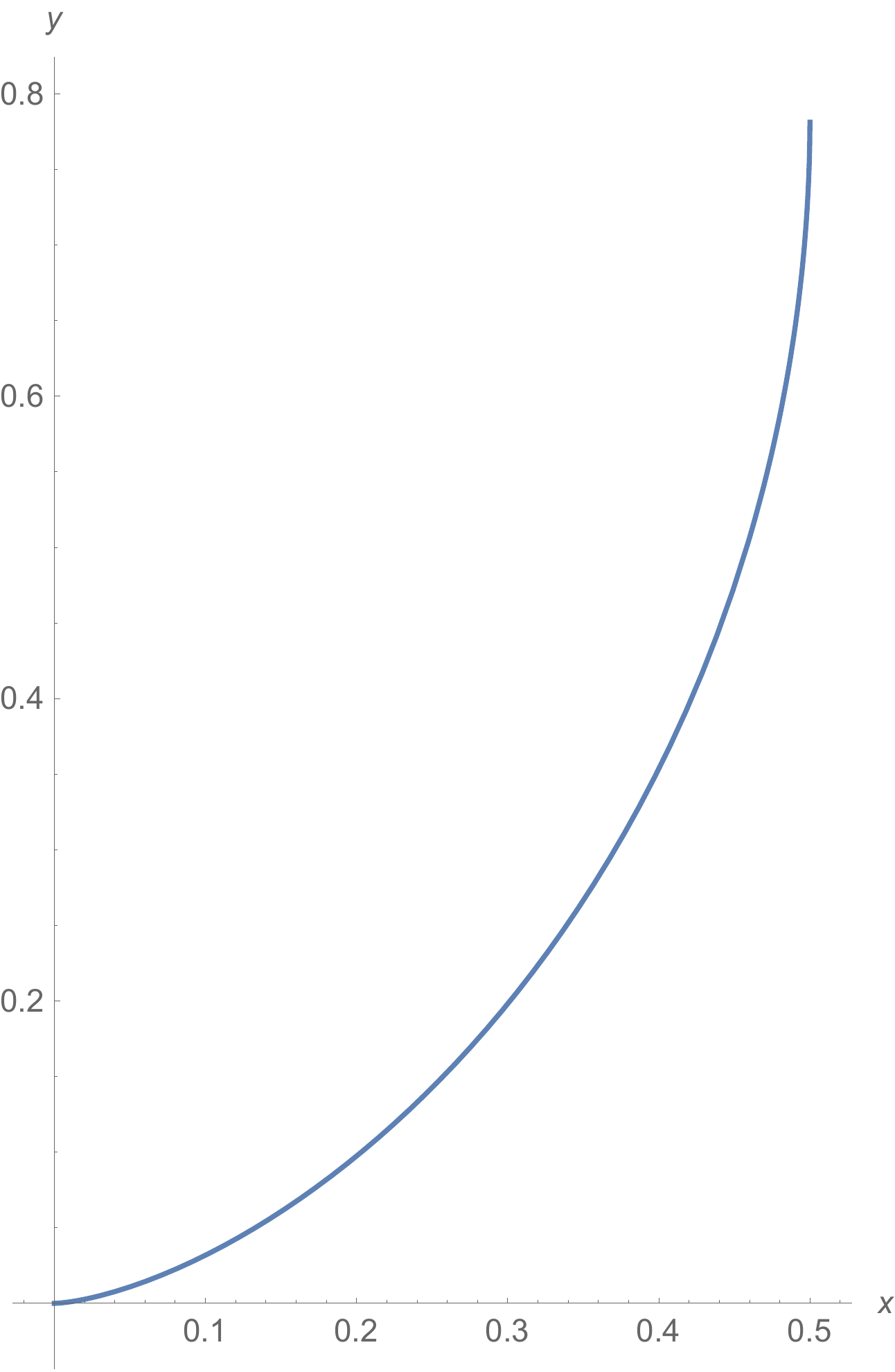} 
\end{center}
\caption{Graphs of \eqref{int-cyl-11} (left)  and \eqref{int-cyl-12} (right). }\label{fig02}
\end{figure}

\subsection{Generalized cylinders $x=ay+g(z)$}

We assume that  $M$ is  a generalized cylinder $x=ay+g(z)$ and the sectional curvature of $M$ is  constant, $K=K_0/2$. Equation \eqref{Kconst-11} is now
$$
(1+a^2+g'^2)(K_0-1)+K_0g'^2\left(1+\frac{g'^2}{1+a^2}\right)+g''=0.
$$
If $g''=0$ identically, then $M$ is a plane of sectional curvature $\frac{1+a^2}{1+a^2+c^2}$, for some constant $c$. Assume $g''\neq 0$. Hence it can be easily shown that $M$ is a grim reaper cylinder as long as $K_0=0$. Let us assume that $K_0\neq0$. We can write
\begin{equation} \label{int531}
2(1+a^2+Q)(K_0-1)+2K_0Q\left(1+\frac{Q}{1+a^2}\right)+Q'=0,
\end{equation}
where the function $Q=g'^2$ is used. We distinguish two cases:

{\bf Case $K_0=1$}. Equation \eqref{int531} is
$$
2Q\left(1+\frac{Q}{1+a^2}\right)+Q'=0.
$$
A first integration yields
\begin{equation} \label{int532}
g'^2=(1+a^2)\frac{c}{e^{2 g}-c}, \quad c>0,
\end{equation}
for $g\in(\log \sqrt{c},\infty)$. Integrating  \eqref{int532}, we find
\begin{equation} \label{sol531}
z=\pm\frac{1}{\sqrt{1+a^2}}(\sqrt{\frac{e^{2 g}}{c}-1}-\tan ^{-1}\left(\sqrt{\frac{e^{2 g}}{c}-1}\right)) .
\end{equation}

{\bf Case $K_0\notin\{0,1\}$}. The solution of \eqref{int531} is
\begin{equation} \label{int533}
g'^2=-(1+a^2)\frac{c+(K_0-1)e^{2 g} }{c+K_0e^{2 g}}, \quad c>0,
\end{equation}
where $K_0<1$ because  the left hand side is strictly positive.  Hence, we have
$$
\left\{ 
\begin{array}{l}
g\in (\log\frac{c}{1-K_0},\log\frac{c}{-K_0} ), \quad K_0<0,  \\ 
g\in (\log\frac{c}{1-K_0},\infty ), \quad 0 <K_0<1. 
\end{array}%
\right. 
$$
Integrating the differential equation \eqref{int533}, one gets
\begin{equation} \label{sol532}
z=\pm\frac{1}{\sqrt{1+a^2}}\int^g\left (\frac{c+K_0e^{2 u}}{-c+(1-K_0)e^{2 u}} \right )^{1/2} du.
\end{equation}
Therefore we have proved

\begin{theorem} \label{t53}
If a generalized cylinder in $\r^3$ of type $x=ay+g(z)$ has constant sectional curvature $K_0/2$ with respect to the canonical snm-connection, then $K_0\leq 1$ and it is either a plane ($K_0=\frac{1+a^2}{1+a^2+c^2}$, $c\in\r$) or a grim reaper cylinder ($K_0=0$), or one of \eqref{sol531} and \eqref{sol532} holds.
\end{theorem}
 
 \begin{remark} Notice that the range for $K$ for generalized cylinders of Thm. \ref{t52} is $\r$ while from Thm. \ref{t53}, the range for $K$ is $(-\infty,1]$.
 \end{remark}

 \begin{remark} Consider the cylinders of types $z=f(x)+ay$ and $x=ay+g(z)$. Assume that $K$ is constant. Without loss of generality we may fix $a=0$, and in such a case their rulings are parallel to the coordinate axis $\partial_y$. Also the generating curves of the two cylinders lie in the $xz$-plane but one is the graph on $x$-axis and the other is on the $z$-axis. This makes a difference, which can be observed by the ranges for $K$ and the integrals given in Thms. \ref{t51} and \ref{t53}.
 \end{remark}

\section*{Acknowledgements}
Rafael L\'opez  is a member of the IMAG and of the Research Group ``Problemas variacionales en geometr\'{\i}a'',  Junta de Andaluc\'{\i}a (FQM 325). This research has been partially supported by MINECO/MICINN/FEDER grant no. PID2020-117868GB-I00,  and by the ``Mar\'{\i}a de Maeztu'' Excellence Unit IMAG, reference CEX2020-001105- M, funded by MCINN/AEI/10.13039/501100011033/ CEX2020-001105-M.

\section*{Author's Statements}
The authors declare that there is no conflict of interest and data availability is not applicable to this article.



\begin{thebibliography}{99}

\bibitem{ag} N. S. Agashe, A semi-symmetric non-metric connection on a Riemannian manifold. Indian J. Pure Appl. Math. 23 (1992), 399–409.

\bibitem{ac} N. S. Agashe, M. R. Chafle, On submanifolds of a Riemannian manifold with a semi-symmetric non-metric connection. Tensor 55 (1994), 120–130.

\bibitem{fs} A. Friedmann, J. A. Schouten, \"{U}ber die Geometrie der halbsymmetrischen \"{U}bertragungen. Math. Z. 21 (1924), 211-223.

 

\bibitem{ha} T. Hasanis, Translation surfaces with non-zero constant mean curvature in Euclidean
space. J. Geom. 110,  (2019), art. 20.

 
\bibitem{hr} T. Hasanis, R. L\'{o}pez, Translation surfaces in Euclidean space with constant Gaussian curvature. Commun. Anal. Geom. 29(6) (2021), 1415--1447.

\bibitem{hr2} T. Hasanis, R. L\'opez,    Classification and construction of minimal translation surfaces in Euclidean space. Results   Math. 75 (2020),  Paper No. 2.



\bibitem{hay} H. Hayden, Subspaces of a space with torsion. Proc. London Math. Soc. 34 (1932), 27–50.

\bibitem{im} T. Imai, Notes on semi-symmetric metric connections. Tensor 24 (1972), 293–296.

\bibitem{li} H. Liu, Translation surfaces with constant mean curvature in 3-dimensional spaces. J. Geom. 64 (1999), 141–149.

\bibitem{lm} R. L\'{o}pez, M. Moruz, Translation and homothetical surfaces in Euclidean
space with constant curvature. J. Korean Math. Soc. 52 (2015), 523-535.

\bibitem{lp} R. L\'{o}pez, O. Perdomo, Minimal translation surfaces in Euclidean space. J. Geom. Anal. 27 (2017), 2926--2937.

\bibitem{lop} R. L\'{o}pez, Invariant singular minimal surfaces. Ann. Global Anal. Geom. 53(4) (2018), 521--541.

\bibitem{lop-grim} R. L\'{o}pez, The translating soliton equation. In: Minimal Surfaces, Integrable Systems and Visualisation, Springer
Proceedings in Mathematics $\&$ Statistics; 2021. pp. 187-216.

\bibitem{am0} A. Mihai, I. Mihai, A note on a well-defined sectional curvature of a semi-symmetric non-metric connection. Int. Electron. J. Geom. 17(1) (2024), 15-23.

\bibitem{na} Z. Nakao, Submanifolds of a Riemannian manifold with semisymmetric metric connections. Proc. Amer. Math. Soc. 54 (1976), 261--266.

\bibitem{sc} H. F. Scherk, Bemerkungen \"{u}ber die kleinste Fl\"{a}che innerhalb gegebener Grenzen. J. Reine Angew. Math. 13 (1835), 185--208.

\bibitem{ya} K. Yano, On semi symmetric metric connection. Rev. Roum. Math. Pures Appl. 15 (1970), 1579--1591.

\end{thebibliography}
 \end{document}